\documentclass[12pt,draftcls,onecolumn]{IEEEtran}

\usepackage[usenames,dvipsnames]{xcolor}
\usepackage{setspace}
\usepackage[draft]{hyperref}
\usepackage{cite}
\usepackage{amsmath,amssymb,euscript,yfonts,psfrag,latexsym,dsfont,graphicx,bbm,color,amstext,wasysym,balance,mathtools}
\usepackage{algorithm,algpseudocode}
\usepackage{graphicx}
\usepackage{amsthm}
\usepackage{float,subfig}

 \usepackage{multirow}

\graphicspath{{../figures/}{figures/}{./figures/}}

\newcommand{\mR}{{\mathbb R}}

\newtheorem{prop}{Proposition}
\newtheorem*{remark}{Remark}

\newcommand{\bv}{{\boldsymbol v}}

\newcommand{\bu}{{\boldsymbol u}}
\newcommand{\bw}{{\boldsymbol w}}
\newcommand{\bx}{{\boldsymbol x}}

\newcommand{\cS}{{\mathcal S}}
\newcommand{\mS}{{\mathbb S}}
\newcommand{\ma}{{\rm a}}
\newcommand{\ms}{{\rm s}}

\newcommand{\As}{{A_{t,\ms}}}
\newcommand{\Aa}{{A_{t,\ma}}}

\newcommand{\cE}{{\mathcal E}}

\newcommand{\g}{{\operatorname{g}}}
\newcommand{\f}{{\operatorname{f}}}
\newcommand{\w}{{\operatorname{w}}}
\newcommand{\trace}{{\operatorname{tr}}}

\begin{document}

\title{Regularization of time-varying covariance matrices using linear stochastic systems}
\author{Lipeng Ning \thanks{This work was supported in part under grants R21MH115280 (PI: Ning), R01MH097979 (PI: Rathi), R01MH111917 (PI: Rathi), R01MH074794 (PI: Westin).
L. Ning is with the Department of Psychiatry, Brigham and Women's Hospital, Harvard Medical School. Email: lning@bwh.harvard.edu} }
	
\maketitle
\date{}

\begin{abstract}
This work focuses on modeling of time-varying covariance matrices using the state covariance of linear stochastic systems.
Following concepts from optimal mass transport and the Schr\"odinger bridge problem (SBP), we investigate several covariance paths induced by different regularizations on the system matrices.
Specifically, one of the proposed covariance path generalizes the geodesics based on the Fisher-Rao metric to the situation with stochastic input. 
Another type of covariance path is generated by linear system matrices with rotating eigenspace in the noiseless situation.
The main contributions of this paper include the differential equations for these covariance paths and the closed-form expressions for scalar-valued covariance.
We also compare these covariance paths using several examples.
\end{abstract}


\IEEEpeerreviewmaketitle

\section{Introduction}\label{sec:intro}

Consider a linear stochastic system 
\begin{align}\label{eq:linsysSto}
\dot \bx_t=A_t \bx_t+\sigma d\bw_t,
\end{align}
where $A_t\in \mR^{n\times n}$ and $\bw_t$ is a standard n-dimensional Wiener process. 
The corresponding state covariance 
$ P_t:=\cE(\bx_t\bx_t')$ evolves according to
\begin{align}\label{eq:sigmat}
\dot  P_t=A_t  P_t+ P_t A_t+\sigma^2I,
\end{align}
where $I$ denotes the identity matrix.
Given two positive definite matrices $ P_0,  P_1$ at $t=0, 1$, respectively, we consider covariance paths that connect the two covariance given by the solution to the following problem
\begin{align}\label{eq:generalProblem}
\min_{ P_t, A_t} \left\{ \int_0^1 f(A_t)dt \mid  \eqref{eq:sigmat} \mbox{ holds, and }  P_0,  P_1 \mbox{ specified}\right\},
\end{align}
where $f(A_t)$ denotes a quadratic function of $A_t$ which may depend on $ P_t$. 
The optimal solutions to \eqref{eq:generalProblem} could provide parametric models of covariance paths which are generalizations of shortest lines or geodesics.
These paths could be applied to fit noisy measurements in order to understand the underlying system dynamics of stochastic processes. 

In this paper, we investigate the solution to \eqref{eq:generalProblem} corresponding to three objective functions, which have been explored in \cite{Ning2018} using linear systems without stochastic input noise.
Specifically, the first type of objective function quantifies the optimal mass transport cost between multivariate Gaussian distributions \cite{VillaniBook,RachevBook,Knott1984,Takatsu2008} when there is no stochastic noise.
In the presence of noise, it has been recently extensively studied that the optimal solution becomes the covariance path induced by the Schr\"odinger Bridge Problem (SBP) \cite{Schrodinger31} \cite{Leonard13,Chen2016Schrodinger}.
The second objective function quantifies a type of weighted-mass-transport cost, which is also related to the Fisher-Rao metric from information geometry \cite{Rao1945,Amari2000}.
The third family of objective function is given by the standard weighted-least-squares of the system matrices.
The main contributions of this paper include the differential equation formulations of the optimal paths induced by the last two objective functions and the closed-form expressions in the special case of scalar-valued covariance.

The organization of this paper is as follows. 
In Section \ref{sec:transport} we revisit covariance paths based on optimal mass transport and SBP.
In Section \ref{sec:info}, we derive the covariance paths corresponding to a weighted-mass-transport cost function, which is equivalent to the Fisher-Rao metric in the noiseless situation.
In Section \ref{sec:rot}, we investigate the covariance paths corresponding to the weighted-least-square functions, which includes the Fisher-Rao based covariance paths in a special case.
In Section \ref{sec:example}, we compare these covariance paths using several examples.
Section \ref{eq:discussion} concludes the paper with discussions.

For notations, $\mS^{n}, \mS^{n}_{+}, \mS^{n}_{++}$ denote the sets of symmetric, positive semidefinite, and positive definite matrices of size $n\times n$, respectively. 
Small boldface letters, e.g. $\bx, \bv$, represent column vectors. Capital letters, e.g. $P, A$, denote matrices. Regular small letters, e.g. $\w, h$ are for scalars or scalar-valued functions.

\section{Mass-transport based covariance paths}\label{sec:transport}

Consider $\bu_t=A_t\bx_t$ as the control input that steers the state covariance matrices according to \eqref{eq:sigmat}. Define
\begin{align*}
\f_{ P_t}^{\rm omt}(A_t)=\cE(\|\bu_t\|_2^2)=\trace(A_t P_t A_t').
\end{align*}
Then the minimum work needed to steer the covariance matrix from $ P_0$ to $ P_1$ is equal to
\begin{align}\label{eq:probOMT}
\min_{ P_t, A_t} \bigg\{\int_0^1 \trace(A_t P_t A_t') dt &\mid \dot  P_t=A_t  P_t+ P_t A_t'+\sigma^2 I,  \nonumber\\
&P_0,  P_1 \mbox{ specified}\bigg\}.
\end{align}
If the noise component vanishes, i.e. $\sigma=0$, then \eqref{eq:probOMT} gives the optimal mass transport distance \cite{VillaniBook,RachevBook,Knott1984,Takatsu2008} between two zero-mean Gaussian distributions with covariance matrices $ P_0$ and $ P_1$, respectively. It is also related to the Bures distance between density matrices in quantum mechanics \cite{Uhlmann1992,Ning2013,Bhatia2017}.
In the general situation when $\sigma$ is non-zero, \eqref{eq:probOMT} can be viewed as the Schr\"odinger Bridge Problem (SBP) \cite{Schrodinger31} between two zero-mean Gaussian probability density functions with covariance being $P_0$ and $P_1$, respectively. 
The basic idea of SBP is to find a stochastic system whose probability law on the diffusion paths is most similar to that of a reference system measured by relative entropy and satisfies the initial and final marginal probability distributions.
Its relations with stochastic optimal control and mass transport have been extensively studied recently \cite{Leonard13,Chen2016Schrodinger}.
The more general situations when the reference model is a linear time-varying system with degenerate diffusions, i.e. noises that influence a subspace of the random states, have also recently been studied in \cite{Chen2016,Chen2016b}.
The following proposition presents the solution to \eqref{eq:probOMT}.

\begin{prop}\label{prop:OMT}
Given $P_0, P_1\in \mS^n_{++}$ and a scalar $\sigma$. The unique set of solution to \eqref{eq:probOMT} is equal to
\begin{align}
A_t^{\rm omt}&=-\Pi_0(I-\Pi_0t)^{-1},\label{eq:Atomt}\\
 P_t^{\rm omt}&=(I-\Pi_0t) P_0(I-\Pi_0t)+\sigma^2(It-\Pi_0t^2),\label{eq:Stomt}
\end{align}
where 
\begin{align}\label{eq:LambdaOMT}
\Pi_0=I -  P_0^{-\tfrac12} \left(\left( P_0^{\tfrac12} P_1 P_0^{\tfrac12}+\frac14 \sigma^4 I \right)^{\tfrac12}-\frac12 \sigma^2 I   \right)  P_0^{-\tfrac12}.
\end{align}
\end{prop}

\begin{proof}
The optimization problem \eqref{eq:probOMT} is viewed as an optimal control problem with $A_t$ being matrix-valued control input. A necessary condition for the optimal solution is that the derivative of the Hamiltonian
\[
h_1( P_t,A_t,\Pi_t)=\trace(A_t P_tA_t+\Pi_t(A_t P_t+ P_t A_t'+\sigma^2 I))
\]
with respect to $A_t$ vanishes with the optimal control. This gives rise to
\begin{align*}
A_t P_t+\Pi_t P_t=0.
\end{align*}
Thus, 
\begin{align}\label{eq:Atomt}
A_t=-\Pi_t.
\end{align}
Next, the optimal $\dot \Pi_t$ also needs to annihilate the derivative of $h_1(\cdot)$ with respect to $ P_t$ which leads to
\begin{align}\label{eq:Lambdatomt}
\dot \Pi_t= -A_t'A_t-A_t'\Pi_t-\Pi_t A_t.
\end{align}
Substituting \eqref{eq:Atomt} to \eqref{eq:sigmat} and \eqref{eq:Lambdatomt} to deduce that
\begin{align}
\dot  P_t&=-\Pi_t P_t- P_t\Pi_t+\sigma^2 I,\label{eq:DotSigma}\\ 
\dot \Pi_t&=\Pi_t^2.\label{eq:DotLambda}
\end{align}

Next, we show that all eigenvalues of the initial $\Pi_0$ must be smaller than one. 
For this purpose, we note that solutions for $ P_t$ and $\Pi_t$ from \eqref{eq:DotSigma} and \eqref{eq:DotLambda} have the following form
\begin{align}
 P_t&=(I-\Pi_0t)  P_0 (I- \Pi_0t)+\sigma^2 (It-\Pi_0t^2),\label{eq:Sigma_omt}\\
\Pi_t&=\Pi_0(I-\Pi_0t)^{-1}.
\end{align}
Assume that $\Pi_0$ has an eigenvalue $\lambda_0 >1$, then $A_t=-\Pi_t$ becomes unbounded when $t$ increases to $1/\lambda_0$. At the same time, $ P_t$ is singular at $t=1/\lambda_0$ and becomes non positive semi-definite when $t> 1/\lambda_0$. 
Therefore, all eigenvalues of $\Pi_0$ are smaller than one.

Next, setting $t=1$ in \eqref{eq:Sigma_omt} and multiplying both sides by $ P_0^{\tfrac12}$ to obtain that
\[
 P_0^{\tfrac12} P_1 P_0^{\tfrac12}=\left(  P_0^{\tfrac12}(I-\Pi_0 )  P_0^{\tfrac12}\right)^2  +\sigma^2  P_0^{\tfrac12}(I-\Pi_0 )  P_0^{\tfrac12}.
\]
Therefore $ P_0^{\tfrac12}(I-\Pi_0 )  P_0^{\tfrac12}$ has the same eigenvectors as $ P_0^{\tfrac12} P_1 P_0^{\tfrac12}$. If $y$ is an eigenvalue of $ P_0^{\tfrac12} P_1 P_0^{\tfrac12}$, then the corresponding eigenvalue of $ P_0^{\tfrac12}(I-\Pi_0 )  P_0^{\tfrac12}$, denoted by $x$, satisfies that 
$x^2+\sigma^2 x=y$. The two solutions are given by 
\[
x_\pm=-\frac12 \sigma^2 \pm(y+\frac14 \sigma^4)^{\tfrac12}.
\]
But $x_{-}$ is negative which contradicts to that $I-\Pi_0$ is positive definite. Thus $x_+$ is the only feasible solution. Therefore, 
\[
 P_0^{\tfrac12}(I-\Pi_0 )  P_0^{\tfrac12}= \left( P_0^{\tfrac12} P_1 P_0^{\tfrac12}+\frac14\sigma^4 I\right)^{\tfrac12}-\frac12 \sigma^2 I,
\]
which gives rise to the optimal solution in \eqref{eq:LambdaOMT}. Then, the proposition is proved.
\end{proof}

We note that if $\Pi_0$ is singular, then $A_t^{\rm omt}$ is also singular for all $t\in [0, t]$, which implies free diffusion in the subspace spanned by the eigenvectors of $\Pi_0$ corresponding to zero eigenvalues. In the noiseless situation when $\sigma=0$, then covariance path in \eqref{eq:Stomt} is equal to the geodesic induced by the Wasserstein-2 metric.
                                                                                                                                                                                                                                                                                                                                                                                                                                                                                                                                                                                                                                                                                                                                                                                                                                                                                                                                                                                                                                                                                                                                                                                                                                                                                                                                                                                                                                                                                                                                                                                                                                                                                                                                                                                                                                                                                                                                                                                                                                                                                                                                                                                                                                                                                                                                                                                                                                                                                                                                                                                                                          \section{Information-geometry based covariance paths}\label{sec:info}
The Fisher information metric has provided a well-defined distance measure between probability distributions. For zero-mean multivariate Gaussian distributions, the Fisher information metric can be expressed as a quadratic form of the covariance matrices, which is referred to as the Fisher-Rao metric \cite{Rao1945,Amari2000}. Specifically, let $ P$ be a positive definite covariance matrix and let $\Delta$ be a symmetric matrix denoting a tangent direction at $ P$ on the manifold of positive-definite matrices. 
Then the Fisher-Rao metric has the following form \cite{Amari2000}
\[
\g_{ P}(\Delta)=\trace( P^{-1}\Delta P^{-1}\Delta).
\]
The geodesic connecting $ P_0,  P_1$ induced by the Fisher-Rao metric is the optimal solution to
\begin{align}\label{eq:probFR}
\min_{ P_t, \dot{ P}_t} \left\{\int_0^1 \trace( P^{-1}\dot{ P}_t P^{-1}\dot{ P}_t) dt \mid   P_0,  P_1 \mbox{ specified}\right\}.
\end{align}
The optimal solution has the following well-known expression \cite{Moakher}
\begin{align}\label{eq:SigmatFR}
 P_t= P_0^{\tfrac12} ( P_0^{-\tfrac12} P_1 P_0^{-\tfrac12})^t  P_0^{\tfrac12}.
\end{align}
In \cite{Ning2018}, we showed that this geodesic is also the optimal solution to an optimization problem in the form of \eqref{eq:generalProblem} in the situation when the noise component $\sigma=0$.
Specifically, we denote
\begin{align*}
\f_{ P_t}^{\rm info}(A_t)=\cE(\|\bu_t\|_{ P_t^{-1}}^2)=\cE(\bu_t'  P_t^{-1} \bu_t)=\trace( P_t^{-1}A_t P_t A_t').
\end{align*}
Then, \eqref{eq:SigmatFR} is also the optimal solution to 
\begin{align}\label{eq:probInfo_noiseless}
\min_{ P_t, A_t} \left\{\int_0^1 \f_{ P_t}^{\rm info}(A_t) dt \mid \dot  P_t=A_t  P_t+ P_t A_t',  P_0,  P_1 \mbox{ specified}\right\},
\end{align}
with the optimal system matrix given by the constant matrix
\[
A=\frac12  P_0^{\tfrac12}\log( P_0^{-\tfrac12} P_1 P_0^{-\tfrac12})  P_0^{-\tfrac12}.
\]
Note that the optimization problem \eqref{eq:probInfo_noiseless} provides an interesting weighted-mass-transport view of the Fisher-Rao metric.

Following \eqref{eq:generalProblem}, we consider the optimization problem in below:
\begin{align}\label{eq:probInfo}
\min_{ P_t, A_t} \bigg\{\int_0^1 \trace( P_t^{-1}A_t P_t A_t') &dt  \mid \dot  P_t=A_t  P_t+ P_t A_t'+\sigma^2 I, \nonumber\\
& P_0,  P_1 \mbox{ specified}\bigg\}.
\end{align}

\subsection{On the optimal paths}
Using variational analysis, we obtain the following proposition on the optimal solution to \eqref{eq:probInfo}.
\begin{prop}\label{prop:info}
Given $P_0, P_1\in \mS^n_{++}$ and a scalar $\sigma$.
If there exists a path that satisfies the following differential equation
\begin{align}
\dot P_t&=-2 P_t \Pi_t  P_t+\sigma^2I,\label{eq:SigmaInfo}\\
\dot \Pi_t&=2\Pi_t P_t \Pi_t,\label{eq:LambdaInfo}
\end{align}
with $P_t$ being equal to $ P_0$ and $P_1$ at $t=0$ and $1$, respectively, then $P_t$ is an optimal solution to \eqref{eq:probInfo}. The corresponding $A_t$ is equal to
\begin{align}\label{eq:AtInfo}
A_t=- P_t\Pi_t.
\end{align}
Moreover, the path $P_t$ also satisfies the following differential equation
\begin{align}\label{eq:GeodesicInfo}
\ddot  P_t -\dot  P_t P_t ^{-1}\dot  P_t+\sigma^4  P_t^{-1}=0.
\end{align}
\end{prop}

\begin{proof}
Following the same method as in the proof of Proposition \eqref{prop:OMT}, we will derive the solution to the optimal control problem \eqref{eq:probInfo} using the following Hamiltonian
\begin{align*}
h_2( P_t,A_t,\Pi_t)=\trace\left(P_t^{-1}A_t P_t A_t' +\Pi_t(A_t P_t+ P_t A_t'+\sigma^2 I)\right).
\end{align*}
It is necessary that $\dot \Pi_t$ annihilates the partial derivative of $h_2(\cdot)$ with respect to $P_t$, which leads to
\begin{align}\label{eq:dotLambda_info}
\dot\Pi_t=-A_t' P_t^{-1}A_t+ P_t^{-1}A_t P_t A_t' P_t^{-1}-\Pi_t A_t-A_t' \Pi_t.
\end{align}
Moreover, setting the derivative of $h_2(\cdot)$ with respect to $A_t$ to zero to obtain that
\begin{align}\label{eq:At_info_condition}
A_t=- P_t\Pi_t.
\end{align}
Then, \eqref{eq:SigmaInfo} and \eqref{eq:LambdaInfo} can be obtained by substituting \eqref{eq:At_info_condition} to \eqref{eq:sigmat} and \eqref{eq:dotLambda_info}, respectively. 
From \eqref{eq:SigmaInfo}, we obtain the following expression
\begin{align*}
\Pi_t=\frac12 (\sigma^2 P_t^{-2}- P_t^{-1}\dot P_t P_t^{-1}).
\end{align*}
Next, taking the derivative of $\Pi_t$ and setting it equal to \eqref{eq:LambdaInfo} we obtain that
\begin{align*}
\dot \Pi_t=&\frac12\bigg(\sigma^2(- P_t^{-1}\dot{ P_t} P_t^{-2}- P_t^{-2}\dot{ P}_t P_t^{-1})\\
&+2 P_t^{-1}\dot{ P}_t P_t^{-1}\dot{ P}_t P_t^{-1}- P_t^{-1}\ddot{ P}_t  P_t^{-1} \bigg)\\
=&\frac12 (\sigma^2 P_t^{-2}- P_t^{-1}\dot P_t P_t^{-1}) P_t (\sigma^2 P_t^{-2}- P_t^{-1}\dot P_t P_t^{-1}).
\end{align*}
Then \eqref{eq:GeodesicInfo} can be obtained from the above equation after simplifications. Thus, the proof is complete.
\end{proof}

In the noiseless situation when $\sigma=0$, \eqref{eq:GeodesicInfo} becomes 
\[
\ddot P_t-\dot{ P}_t P_t^{-1}\dot{ P}_t=0,
\]
which is the geodesic equation induced by the Fisher-Rao metric \cite{Moakher,Lenglet2006}. In this case, the closed-form expression of $P_t$ is given by \eqref{eq:SigmatFR}.
The closed-form expression of the optimal solution to \eqref{eq:probInfo} is currently unknown to the author except for the scalar-valued cases given in below.

\subsection{The scalar case}
Consider a scalar-valued covariance path $p_t$ that is expressed by one of the following equations:
\begin{align}
p_t&=p_0+\sigma^2 t,\label{eq:solution1}\\
p_t&=ae^{bt}-\frac{\sigma^4}{4ab^2}e^{-bt},\label{eq:solution2}\\
p_t&=\frac{\sigma^2}{\omega}\cos(\omega t+\theta),\label{eq:solution3}\\
& \mbox{ with } 0<\omega< \pi, -\tfrac{\pi}{2}<\theta<\tfrac{\pi}{2}. \nonumber
\end{align}
It is straightforward to verify that all the above expressions satisfy that 
\[
p_t\ddot p_t -(\dot p_t)^2+\sigma^4=0,
\]
which is a scalar-version of \eqref{eq:GeodesicInfo}.
Clearly, \eqref{eq:solution1} corresponds to a covariance path purely driven by the input noise.
For the covariance path in \eqref{eq:solution3}, the constraint $0<\omega<\pi$ is to ensure that $P_t$ does not have negative values on $t\in[0, 1]$. 
Moreover, the constraint $-\frac{\pi}{2}<\theta<\frac{\pi}{2}$ guarantees that $p_0$ is positive.
The following proposition shows that \eqref{eq:solution2} and \eqref{eq:solution3} connect covariance that satisfy different boundary conditions.

\begin{prop}\label{prop:expsolution}
Given three positive scalars $p_0, p_1$ and $\sigma$. If $|p_1-p_0|>\sigma^2$, then exists a unique solution to \eqref{eq:probInfo} which has the form of \eqref{eq:solution2}. If $|p_1-p_0|<\sigma^2$, then there exists a unique solution to \eqref{eq:probInfo} which has the form of \eqref{eq:solution3}.
\end{prop}
\begin{proof}
Since the paths in \eqref{eq:solution2} and \eqref{eq:solution3} both satisfy the \eqref{eq:GeodesicInfo}.
We will prove that if $|p_0-p_1|>\sigma^2$, then these exists a unique pair of parameters $a, b$ such that $p_t$ from \eqref{eq:solution2} is equal to $p_0$ and $p_1$ at $t=0, 1$, respectively and there exists no path of the form \eqref{eq:solution3} that satisfies the boundary conditions.
If $|p_1-p_0|<\sigma^2 $ then there exists a unique path of the form \eqref{eq:solution3} and there exists no path of the form \eqref{eq:solution2} that satisfies the boundary conditions.

To prove the first statement, we set $p_t$ in \eqref{eq:solution2} to $p_0$ and $p_1$ at $t=0, 1$, respectively, to obtain that
\begin{align}
p_0&=a-\frac{\sigma^4}{4ab^2},\label{eq:P0ab}\\
p_1&=a e^{b}-\frac{\sigma^4}{4ab^2}e^{-b}.\nonumber
\end{align} 
Then, we can derive the following
\[
a=\frac{p_1-p_0e^{-b}}{e^{b}-e^{-b}}.
\]
Next, substituting the above expression into \eqref{eq:P0ab} then multiplying both sides by $4(p_1-p_0e^{-b})b^2$ to obtain that
\begin{align}\label{eq:hb}
4(p_1-p_0e^{-b})(p_1-p_0e^b)b^2-\sigma^4(e^b-e^{-b})^2=0.
\end{align}
The above equation has a trivial solution at $b=0$. Moreover, if $b$ is a solution to \eqref{eq:hb} so is $-b$.
For a non-zeros $b$ that satisfies \eqref{eq:hb}, the coefficient of the second term of $p_t$ is equal to
\[
-\frac{\sigma^4}{4ab^2}=\frac{p_1-p_0e^b}{e^{-b}-e^{b}}.
\]
Therefore, \eqref{eq:solution2} is equivalent to
\[
p_t=\frac{p_1-p_0e^{-b}}{e^{b}-e^{-b}} e^{bt}+\frac{p_1-p_0e^b}{e^{-b}-e^{b}}e^{-bt}.
\]
Thus, switching $b$ to $-b$ does not change the covariance path in \eqref{eq:solution2}. 
Therefore, we only consider positive solutions to \eqref{eq:hb}.
To this end, we denote the left hand side of \eqref{eq:hb} by $\phi(b)$. It is straightforward to derive that 
\begin{align*}
\frac{d \phi(b)}{db}\mid_{b=0}&=0, \\
\frac{d^2 \phi(b)}{db^2}\mid_{b=0}&=8((p_1-p_0)^2-\sigma^4).
\end{align*}
Moreover, 
\[
\frac{d^3 \phi(b)}{db^3}=-4p_0p_1(e^b-e^{-b})(b^2+6b+6)-8\sigma^4(e^{2b}-e^{-2b})
\] 
which is negative for all $b>0$. 
Therefore, if $(p_1-p_0)^2-\sigma^4>0$, then the function $\phi(b)$ is convex and positive near $b=0$.
But $\phi(b)<0$ as $b\rightarrow \infty$. Because its second order derivative $\frac{d^2 \phi(b)}{db^2}$ is monotonically decreasing, we conclude that there exist a unique solution to $\phi(b)=0$.

To prove the second statement, we take the derivative of $p_t$ given by \eqref{eq:solution3} to obtain
\begin{align*}
\dot p_t=-\sigma^2 \sin(\omega t+\theta).
\end{align*}
If $p_0, p_1$ are the endpoints of $p_t$, then it is necessary that
\[
|p_1-p_0|=|\int_0^1 \dot p_t dt |\leq \int_0^1 |\dot p_t| dt \leq \sigma^2.
\]
To show that there exist a path of the form \eqref{eq:solution3} for any $p_0, p_1$ that satisfies $|p_1-p_0|<\sigma^2$, we re-parameterize \eqref{eq:solution3} as follows
\begin{align*}
p_t=c\cos(\omega t)+d\sin(\omega t),
\end{align*}
with $(c^2+d^2)\omega^2=\sigma^4$. By setting $p_t$ to $p_0, p_1$ at $t=0, 1$, respectively, we obtain that
\begin{align}
p_0&=c,\label{eq:P0c}\\
p_1&=c\cos(\omega)+d\sin(\omega).\nonumber
\end{align}
Therefore, 
\begin{align}\label{eq:P1d}
d=\frac{p_1-p_0\cos(\omega)}{\sin(\omega)}.
\end{align}
Next, substituting \eqref{eq:P0c} and \eqref{eq:P1d} to $\sin^2(\omega)(c^2+d^2)=\sin^2(\omega)\frac{\sigma^4}{\omega^2}$ to obtain that 
\begin{align}\label{eq:omegaeq}
\frac{\sigma^4}{\omega^2}\sin^2(\omega)+2p_0p_1\cos(\omega)=p_0^2+p_1^2.
\end{align}
We denote the left hand size of the above equation by $\psi(\omega)$. 
Then, the following holds
\[
\lim_{\omega \rightarrow 0} \psi(\omega)=\sigma^4+2p_0p_1\geq (p_1-p_0)^2+2p_0p_1=p_0^2+p_1^2.
\]
Moreover, $\psi(\pi)=-2p_0p_1<0$. 
Furthermore, it can be shown that 
\[
\frac{d\psi(\omega)}{d\omega}=\left(2\frac{\sigma^4}{\omega^3}(\omega\cos(\omega)-\sin(\omega))-2p_0p_1 \right)\sin(\omega)<0
\]
for all $\omega\in (0, \pi)$. Therefore, there exists a unique solution to \eqref{eq:omegaeq} for $\omega\in(0, \pi)$, which completes the proof.
\end{proof}

\section{Weighted-least-squares based covariance paths}\label{sec:rot}
For scalar-valued covariance, the Fisher-Rao metric $\g_{ p}(\dot  p)$ can be viewed as the squared norm $(\tfrac{\dot  p}{ p})^2$.
For matrices, there is in general no unique ways to define matrix divisions. For example, in the following differential equation
\begin{align}\label{eq:dSAt}
\dot  P_t=A_t P_t+ P_t A_t',
\end{align}
the matrix $A_t$ can be viewed as a non-commutative division of $\frac12\dot P$ by $ P$. 
For a given pair of matrices $ P_t$ and $\dot{ P}_t$ there are infinite $A_t$ that satisfy \eqref{eq:dSAt}. 
But there could be a unique one that minimizes or maximizes a quadratic function $\f(A_t)$ whose optimal value provides a well-defined quadratic norm of the non-commutative division of $\dot P_t$ by $ P_t$. The optimal value of the quadratic function provides a way to define Riemmanian metrics on the manifold of covariance matrices in order to quantify the similarity between covariance matrices. This motivates our choice of the third objective function.

Specifically, for a matrix $A_t\in \mR^{n\times n}$, we decompose it as $A_t=\As+\Aa$ where
\begin{align*}
A_{t,\ms}&:=\tfrac12 (A_t+A_t'),\\
A_{t,\ma}&:=\tfrac12 (A_t-A_t'),
\end{align*}
are the symmetric and asymmetric parts of $A_t$, respectively.
For a given scalar $\epsilon> 0$, we define the following weighted squared norm of $A_t$
\begin{align*}
\f_{\epsilon}^{\rm wls}(A_t)&:=\|A_{t,\ms}\|_{\rm F}^2+\epsilon\|A_{t,\ma}\|_{\rm F}^2\\
&=\frac{1+\epsilon}{2}\trace(A_tA_t')+\frac{1-\epsilon}{2}\trace(A_tA_t).
\end{align*}

Following \eqref{eq:generalProblem}, we consider the optimization problem in below:
\begin{align}\label{eq:probRot}
\min_{ P_t, A_t} \bigg\{&\int_0^1 \frac{1+\epsilon}{2}\trace(A_tA_t')+\frac{1-\epsilon}{2}\trace(A_tA_t) dt \mid\nonumber \\
& \dot  P_t=A_t  P_t+ P_t A_t'+\sigma^2 I,  P_0,  P_1 \mbox{ specified}\bigg\}.
\end{align}

\begin{prop}\label{prop:Rot}
Given $P_0, P_1\in \cS^n_{++}$ and two scalars $\epsilon,\sigma>0$. 
If there exists a pair of matrix-valued functions $P_t, \Pi_t$ that satisfy 
\begin{align}
\dot P_t&=-\frac{1+\epsilon}{2\epsilon} \left(\Pi_t  P_t^2+ P_t^2\Pi_t \right) +\frac{1-\epsilon}{\epsilon} P_t\Pi_t P_t+\sigma^2I,\label{eq:SigmaRot}\\
\dot \Pi_t&=\frac{1+\epsilon}{2\epsilon}(\Pi_t^2 P_t+ P_t\Pi_t^2)-\frac{1-\epsilon}{\epsilon} \Pi_t P_t\Pi_t,\label{eq:LambdaRot}
\end{align}
with $P_t$ being equal to $P_0, P_1$ at $t=0, 1$, respectively, then $P_t$ is a optimal solution to \eqref{eq:probRot}.
Moreover, the optimal $A_t$ is given by
\begin{align}\label{eq:AtRot}
A_t=-\frac12(\Pi_t P_t+ P_t\Pi_t)+\frac{1}{2\epsilon}( P_t\Pi_t-\Pi_t P_t).
\end{align}
\end{prop}

\begin{proof}
Following the same method as in the proof of Proposition \eqref{prop:info}, we develop the necessary conditions for a stationary value using the following Hamiltonian
\begin{align*}
h_3( P_t,A_t,\Pi_t)=&\trace\bigg(\frac{1+\epsilon}{2}\trace(A_tA_t')+\frac{1-\epsilon}{2}\trace(A_tA_t) \\
&+\Pi_t(A_t P_t+ P_t A_t'+\sigma^2I)\bigg).
\end{align*}
Setting $-\dot \Pi_t$ equal to partial derivative of $h_3(\cdot)$ with respect to $ P_t$ to obtain that
\begin{align}\label{eq:dotLambda_rot}
\dot\Pi_t=-\Pi_t A_t-A_t' \Pi_t.
\end{align}
Moreover, setting the partial derivative of $h_2(\cdot)$ with respect to $A_t$ equal to zero to obtain that 
\begin{align}\label{eq:At_rot_condition}
(1+\epsilon)A_t+(1-\epsilon) A_t'+2 \Pi_t P_t=0.
\end{align}
Thus, the symmetric and asymmetric part of $A_t$ are equal to
\begin{align*}
\As&=-\frac12(\Pi_t P_t+ P_t\Pi_t),\\
\Aa&=\frac{1}{2\epsilon}( P_t\Pi_t-\Pi_t P_t).
\end{align*}
Therefore \eqref{eq:AtRot} holds. Then, \eqref{eq:SigmaRot} and \eqref{eq:LambdaRot} can be obtained by substituting \eqref{eq:AtRot} into \eqref{eq:sigmat} and \eqref{eq:dotLambda_rot}, respectively.
\end{proof}
Moreover, the path defined by \eqref{eq:SigmaRot} and \eqref{eq:LambdaRot} also coincide with the path given by \eqref{eq:SigmaInfo} and \eqref{eq:LambdaInfo} for scalar-valued covariance. Therefore, the results from Proposition \ref{prop:expsolution} can also be adapted to the solution of \eqref{eq:probRot}.

\begin{prop}\label{prop:expsolution_rot}
Given three positive scalars $p_0, p_1$ and $\sigma$. If $|p_1-p_0|>\sigma^2$, then exists a unique solution to \eqref{eq:probRot} which has the form of \eqref{eq:solution2}. If $|p_1-p_0|<\sigma^2$, then there exists a unique solution to \eqref{eq:probRot} which has the form of \eqref{eq:solution3}.
\end{prop}

\begin{remark}
For any given $\epsilon\neq0$ and a pair of initial values $\Pi_0, P_0$, the pair of equations \eqref{eq:SigmaRot} and \eqref{eq:LambdaRot} provide a smooth path of covariance matrices. In the special case when $\epsilon=-1$, \eqref{eq:SigmaRot} and \eqref{eq:LambdaRot} is equivalent to the Fisher-Rao based path given by \eqref{eq:SigmaInfo} and \eqref{eq:LambdaInfo}. In the noiseless situation, the existence and uniqueness of the covariance path that satisfy \eqref{eq:SigmaRot} and \eqref{eq:LambdaRot} has been shown in \cite{Ning2018} for $\epsilon$ taking values around $-1$.

\end{remark}

To understand the influence of input noise to the optimal covariance path, we consider a trajectory defined by \eqref{eq:SigmaRot} to \eqref{eq:AtRot} with given initial values $ P_0$ and $\Pi_0$ and the symmetric and asymmetric parts of the system matrix $A_t$ are initialized by \eqref{eq:AtRot}. Then it is straightforward to verify that the asymmetric part of $A_t$ given by \eqref{eq:AtRot} is constant, which is denoted by $A_{0,\ma}$ and is equal to $\frac{1}{2\epsilon}( P_0\Pi_0-\Pi_0 P_0)$.
On the other hand, by taking the derivative of the symmetric part of $A_t$, denoted by $A_{t,\ms}$, we obtain that
\begin{align}\label{eq:dotAts}
\dot A_{t,\ms}=(1+\epsilon) (A_{0,\ma} A_{t,\ms}+A_{t,\ms}A_{0,\ma}')-\sigma^2\Pi_t.
\end{align}
Next, we apply change of variables to define
\[
\hat A_t=e^{(1+\epsilon)A_{0,\ma}'t} \left( A_{t,\ms}+\epsilon A_{0,\ma}'\right) e^{(1+\epsilon)A_{0,\ma} t}.
\]
Similarly, applying a time-varying change of coordinate to define
\begin{align*}
\hat  P_t&=e^{(1+\epsilon)A_{0,\ma}'t}  P_t e^{(1+\epsilon)A_{0,\ma} t},\\
\hat \Pi_t&=e^{(1+\epsilon)A_{0,\ma}'t} \Pi_t e^{(1+\epsilon)A_{0,\ma} t}.
\end{align*}
Then the derivative of the new variable $\hat A_t$ is equal to
\begin{align}\label{eq:dothatAt}
 \dot{\hat{A}}_t=-\sigma^2 \hat \Pi_t.
\end{align}
Moreover, the derivative of $\hat  P_t$ is equal to
\begin{align}
\dot{\hat{ P}}_t=& (1+\epsilon)A_{0,\ma}'t \hat  P_t+(1+\epsilon)\hat  P_t A_{0,\ma}\nonumber\\
&+ e^{(1+\epsilon)A_{0,\ma}'t} (A_t P_t+ P_tA_t'+\sigma^2 I) e^{(1+\epsilon)A_{0,\ma} t}\nonumber\\
=&\hat A_t \hat  P_t+\hat P_t \hat A_t'+\sigma^2 I.\label{eq:hatSigma_rot}
\end{align}
Similarly, the derivative of $\hat{\Pi}_t$ is given by
\begin{align}
\dot{\hat{\Pi}}_t&= -\hat A_t' \hat \Pi_t-\hat\Pi_t \hat A_t.\label{eq:hatLambda_rot}
\end{align}
Combining \eqref{eq:dothatAt} and \eqref{eq:hatLambda_rot}, we obtain that the covariance path is determined by the following second-order differential equation in the roating frame
\begin{align}\label{eq:ddotAt}
\ddot{\hat{A}}_t+\hat A_t' \dot{\hat{A}}_t+\dot{\hat{A}}_t'\hat A_t=0,
\end{align}
with the initial values given by
\begin{align}
\hat A_0&=A_{0,\ms}+\epsilon A_{0,\ma}'=- P_0\Pi_0,\label{eq:hatA0}\\
\dot{\hat{A}}_0&=-\sigma^2 \Pi_0.\label{eq:dothatA0}
\end{align}

The special case when $\sigma=0$, the matrix $\hat A_t$ is clearly constant based on \eqref{eq:dothatAt}. 
Therefore, \eqref{eq:hatSigma_rot} becomes a linear time-invariant system whose solution is given by the standard expression $\hat  P_t=e^{\hat A_0t} P_0e^{\hat A_0't}$. Then, transforming $\hat A_t$ and $\hat  P_t$ to the original coordinate to obtain that
\begin{align}
A_t&=e^{(1+\epsilon)A_{0,\ma} t} A_0 e^{(1+\epsilon)A_{0,\ma} 't},\label{eq:At_rot_sigma0}\\
 P_t&=T_{\epsilon,t}(A) P_0T_{\epsilon,t}(A)',\label{eq:Sigma_rot_sigma0}
\end{align} 
where $T_{\epsilon,t}(A)=e^{(1+\epsilon)A_{0,\ma} t}e^{(A_{0,\ms}+\epsilon A_{0,\ma}')t}, A_{0,\ms}=-\frac12(\Pi_0 P_0+ P_0\Pi_0), A_{0,\ma}=\frac{1}{2\epsilon}( P_0\Pi_0-\Pi_0 P_0)$. $T_{\epsilon,t}(A)$ is the state transition matrix which satisfies that $\dot {T}_{\epsilon,t}(A)=A_t T_{\epsilon,t}(A)$. It is interesting to note that the system matrix $A_t$ has a rotating eigenspace. If $\epsilon$ is sufficiently close to $-1$, then there exists a unique covariance path of the form \eqref{eq:dothatA0} that connects any given $P_0, P_1\in \mS^n_{++}$ \cite{Ning2018}.

\section{Example}\label{sec:example}
\subsection{Comparing scalar-valued covariance paths}
In this example, we compare scalar-valued covariance paths given by the closed-form expressions in \eqref{eq:Stomt} and  \eqref{eq:solution2} to \eqref{eq:solution3}.
Figure \ref{fig:scalar} illustrates several covariance paths with the initial value being $P_0=6$ and $\sigma=4$. These plots clearly illustrate the differences between the paths obtained from OMT and the Fisher-Rao metric, especially between the paths whose endpoints at $t=1$ are far away from $P_0+\sigma^2$, which is equal to $22$ in this example.
\begin{figure*}[htb]
\centering
\includegraphics[width=.5\textwidth]{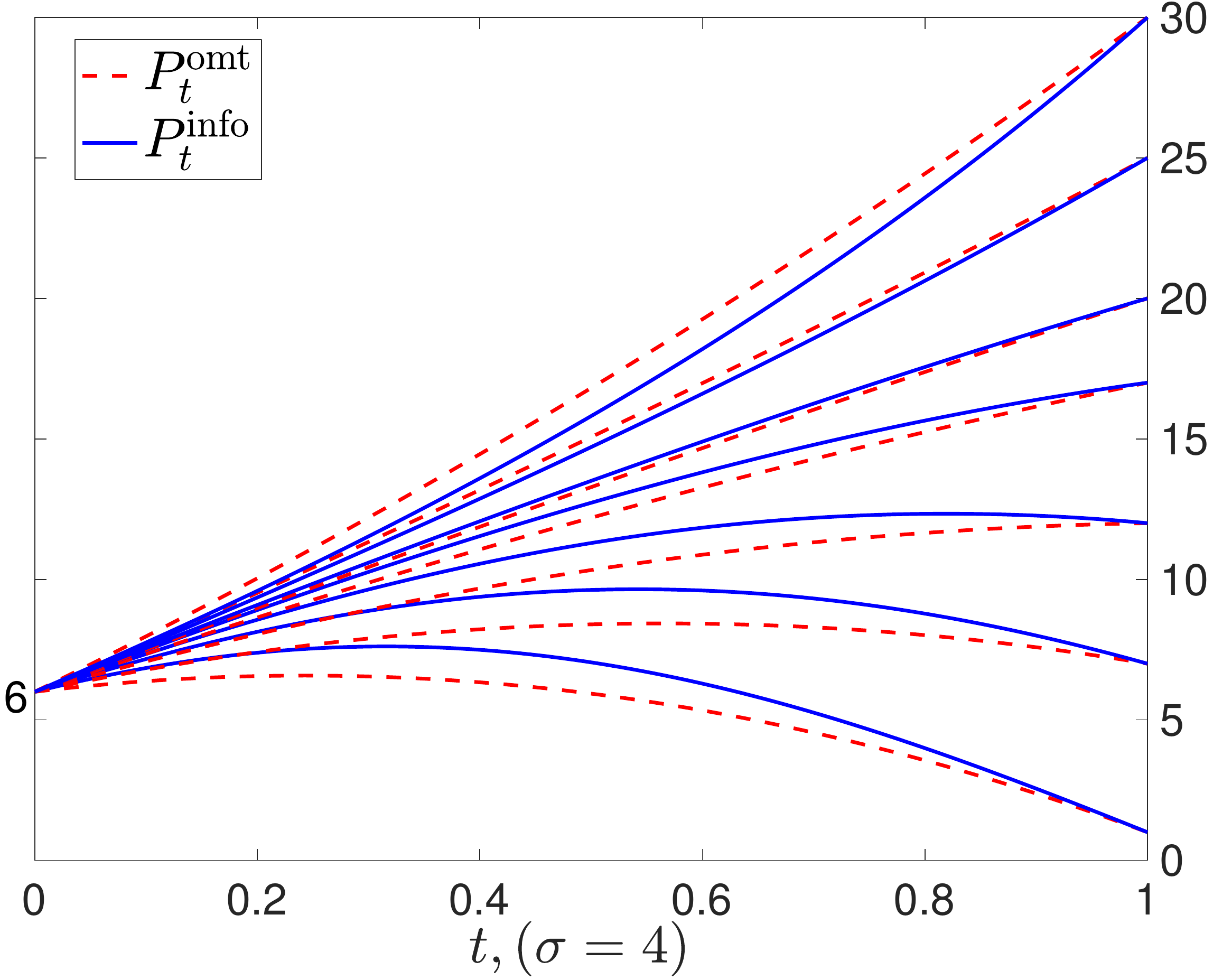}
\caption{\small{An illustration of the covariance paths. The dashed red lines denoted by $P_t^{\rm omt}$ illustrate the paths given by \eqref{eq:Stomt} based on OMT. The slid blue lines denoted by $P_t^{\rm info}$ illustrate the covariance paths given by \eqref{eq:solution2} to \eqref{eq:solution3} based on the Fisher-Rao metric.}}\label{fig:scalar}
\end{figure*}

\subsection{Comparing matrix-valued covariance paths}
In this example, we illustrate the difference between the above covariance paths using two fixed endpoints given by 
\begin{align}\label{eq:S0S1}
 P_0= \left[\begin{matrix}1 & 0 \\ 0& 0.3 \end{matrix}\right],  P_1= \left[\begin{matrix}0.3 & 0 \\ 0& 1 \end{matrix}\right]. 
\end{align}
Fig. \ref{fig:OMT2D} shows the OMT-based paths in \eqref{eq:Stomt} using several different values for $\sigma$, where the ellipsoids denote isocontour of the quadratic function 
$\bx' P_t \bx=r^2$ with $r=0.5$. 
Fig. \ref{fig:Info2D} illustrates the Fisher-Rao-based paths given by the \eqref{eq:GeodesicInfo}. Since $P_0$ and $P_1$ commute, the paths are obtained by using the closed-form expressions in \eqref{eq:solution2} to \eqref{eq:solution3} to the diagonal entries. Though there are minor differences between the two sets of covariance paths, all the $P_t$'s along the two paths have the same eigenspace.

Regarding the covariance paths defined by \eqref{eq:SigmaRot} and \eqref{eq:LambdaRot}, their closed-form solutions are currently unknown. For this particular pair of end points, there exists multiple local optimal paths. 
Specifically, we consider the solution to \eqref{eq:probRot} when $\epsilon=0$. In this case, any asymmetric system matrix $A$ of the form
\begin{align}\label{eq:A}
A=\left[\begin{matrix} 0& \pm \tfrac{(2k+1)\pi}{2}\\ \mp \tfrac{(2k+1)\pi}{2}& 0 \end{matrix} \right],
\end{align}
is an optimal solution because the corresponding objective value is equal to zero. 
The corresponding covariance paths are of the form
{\begin{align*}
 P_{\pm,t}=&\left[\begin{matrix}0.3+0.7\cos^2(\tfrac{(2k+1)\pi}{2}t)& \pm 0.7\cos(\tfrac{(2k+1)\pi}{2}t)\sin(\tfrac{(2k+1)\pi}{2}t)\\ \pm 0.7\cos(\tfrac{(2k+1)\pi}{2}t)\sin(\tfrac{(2k+1)\pi}{2}t)& 0.3+0.7\sin^2(\tfrac{(2k+1)\pi}{2}t) \end{matrix} \right]+\sigma^2 It.
\end{align*}}

In order to obtain a numerical solution for the covariance paths given by \eqref{eq:SigmaRot} and \eqref{eq:LambdaRot} with $\epsilon>0$, we apply the \emph{lsqnonlin} nonlinear optimization toolbox and the \emph{ode45} functions in MATLAB (The MathWorks, Inc., Natic, MA) to solve for $\Pi_0$ so that a path $P_t$ starts from $P_0$ will have the least square error relative to $P_1$ at $t=1$. We first set $\epsilon=0.001$ and choose two different initial values for $\hat \Pi_0$ used in the optimization algorithms given by 
\[
\hat \Pi_{\pm,0}=\pm\frac{1}{700}\left[\begin{matrix}0 & \pi \\ \pi & 0\end{matrix} \right],
\]
so that the corresponding system matrices given by \eqref{eq:AtRot} approximately satisfy \eqref{eq:A}. Next, we gradually increase $\epsilon$ using a step size of $0.001$ and use the optimal $\Pi_0$ from the previous step as the initial value. The left and right panels of Fig. \ref{fig:Rot} illustrate two branches of locally optimal paths corresponding to the two different initial values of $\Pi_0$ with $\sigma=0.5$ and several different values for $\epsilon$. All the numerical solutions satisfy the endpoints with the Frobenius norm of the residuals at the order of $10^{-6}$ or smaller. Very different from the paths shown in \eqref{fig:omt}, the paths in Fig. \ref{fig:Rot} have rotating eigenspace where the rotation direction depends on the initial choice of $\Pi_0$. Moreover, as $\epsilon$ increase, the paths become similar to those shown in Fig. \ref{fig:omt}.

\begin{figure*}[htb]
\centering
\subfloat[][OMT-based paths]{\includegraphics[width=.45\textwidth]{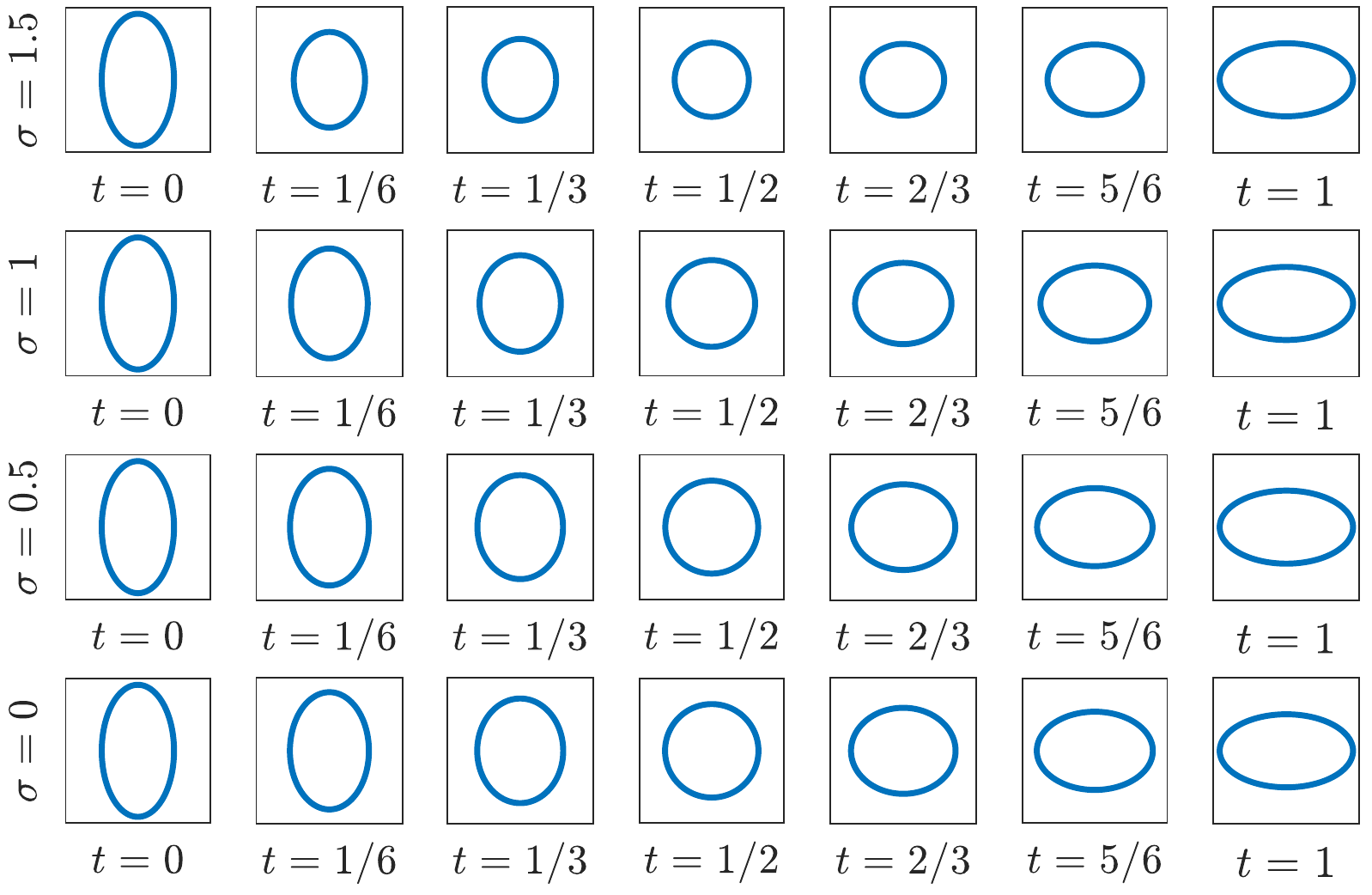}\label{fig:OMT2D}}\quad
\subfloat[][Fisher-Rao-based paths]{\includegraphics[width=0.45\textwidth]{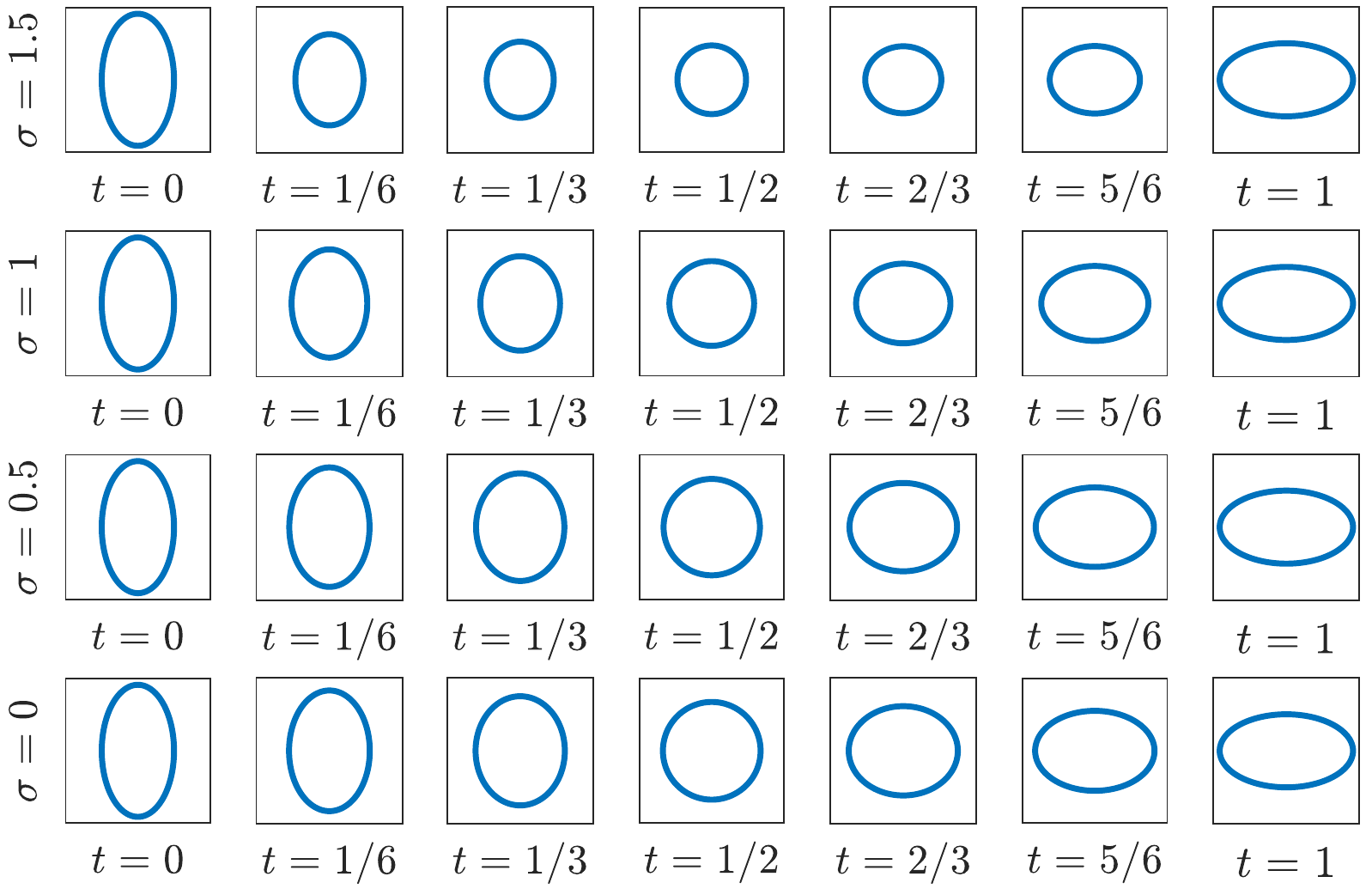}\label{fig:Info2D}}
\caption{\small{An illustration of the covariance paths obtained using \eqref{eq:Stomt} (Left) and \eqref{eq:GeodesicInfo} (Right), respectively, with the two endpoints given by $ P_0$ and $ P_1$ in \eqref{eq:S0S1}.}}\label{fig:omt}
\end{figure*}

\begin{figure*}[htb]
\centering
\subfloat[][]{\includegraphics[width=.45\textwidth]{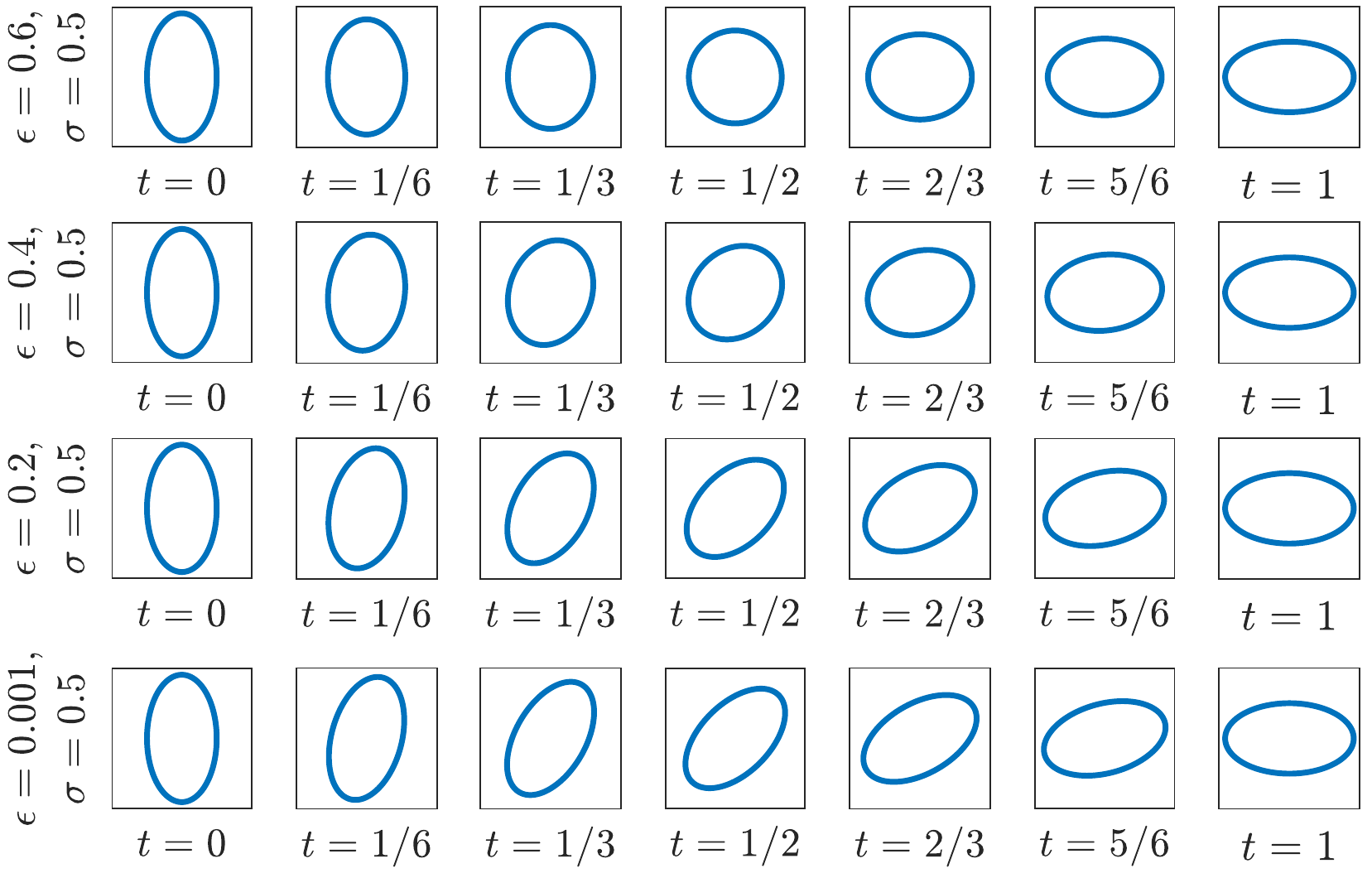}}\quad
\subfloat[][]{\includegraphics[width=0.45\textwidth]{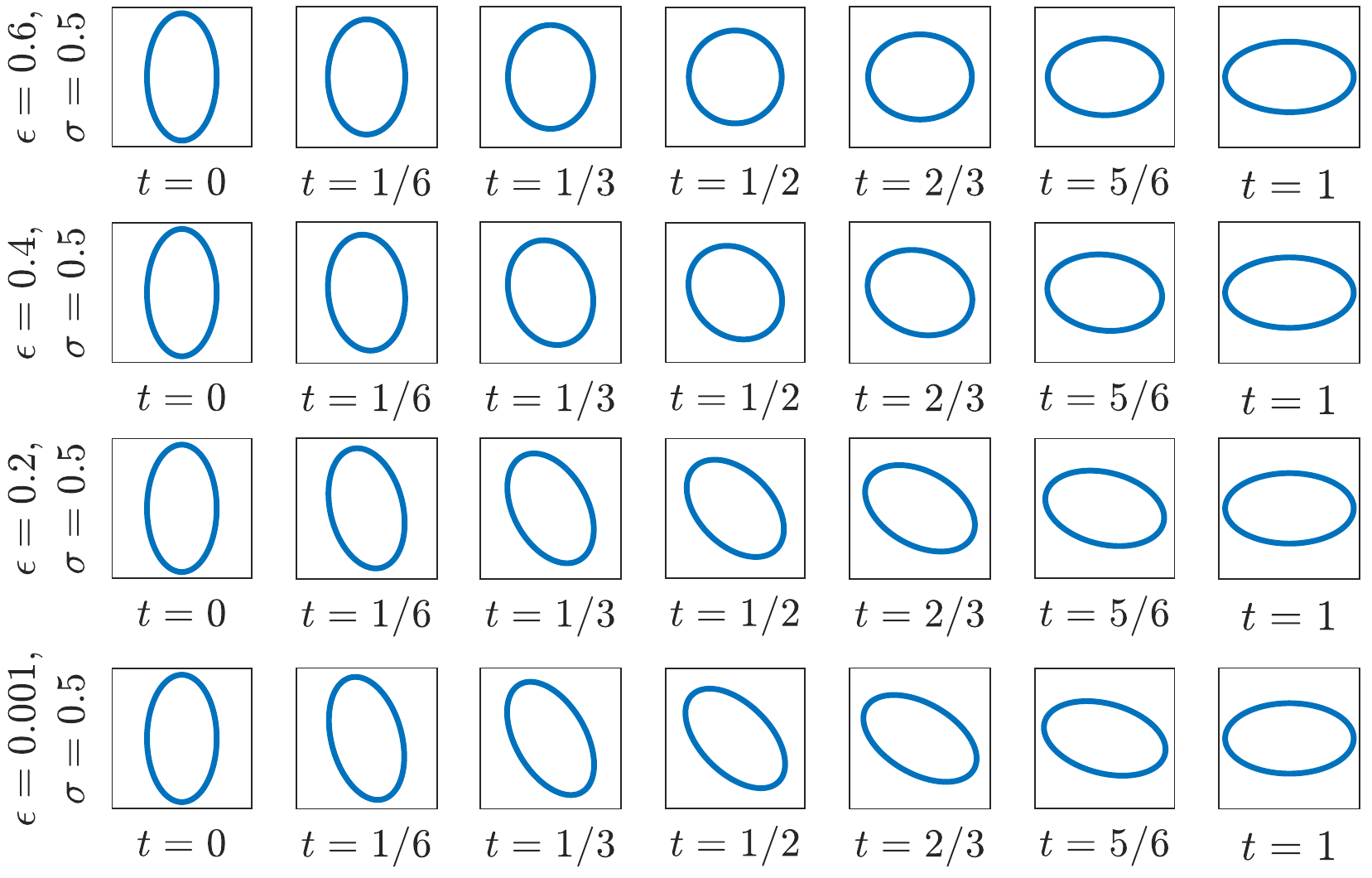}}
\caption{\small{An illustration of two branches, i.e. (a) and (b), of locally-optimal covariance paths obtained using \eqref{eq:SigmaRot} and \eqref{eq:LambdaRot} with different initial values for $\Pi$.}}\label{fig:Rot}
\end{figure*}

\subsection{Fitting noisy measurements of functional MRI data}
In this example, we apply the proposed covariance paths to fit noisy sample covariance matrices of a stochastic process based on a resting-state functional MRI (rsfMRI) dataset used in \cite{Ning2018}. 
The interested reader is referred to \cite{Ning2018} for more detailed information on the data. 
This problem is formulated as follows. 
Given $K$ sample covariance matrices, $\tilde  P_{t_1}, \ldots, \tilde P_{t_K}$ based on $K$ segments of a multivariate time series, find a smooth covariance path $ P_t$ that minimizes
\begin{align}\label{eq:probB}
\min_{ P_t}\sum_{k=1}^K \| P_{t_k}-\tilde P_{t_k}\|_{\rm F}^2.
\end{align}
In this example, we has $K=10$ noisy sample covariance from a 7-dimensional process. 
We apply three parametric models of covariance path to fit these measurements using the \emph{fminsdp} function in MATLAB. 
The first one is the closed-form expression given by \eqref{eq:Stomt} which is parameterized by $ P_0, \Pi_0$ and $\sigma$. 
The optimal solution is denoted by $\hat  P_{t}^{\rm omt}$. 
The second model is based on the differential equation \eqref{eq:SigmaInfo} and \eqref{eq:LambdaInfo} which is implemented by the \emph{ode45} function in MATLAB. The estimated path is denoted by $\hat P_t^{\rm info}$. 
Note again that this model is a special case of \eqref{eq:SigmaRot} and \eqref{eq:LambdaRot} when $\epsilon=-1$. 
The more general model in \eqref{eq:SigmaRot} and \eqref{eq:LambdaRot} relies on the estimation an additional parameter $\epsilon$. 
Since the differential equations are highly nonlinear in term of $\epsilon$, we find that the MATLAB algorithm only provides a local optimal value of $\epsilon$ depending on its initial value. 
In order to obtain a reliable covariance path to understand the rotation of energy among the variables, we apply the rotating-system based covariance path in \eqref{eq:At_rot_sigma0} to fit the measurements by setting $\epsilon=20$ as used in \cite{Ning2018}. The estimated path is denoted by $\hat  P_t^{\rm wls}$.

The discrete makers in Fig. \ref{fig:Fit} illustrate the noisy sample covariance of 6 entries of the covariance matrices. The blue, green and red plots are the estimated paths given by $\hat  P_t^{\rm omt}, \hat  P_t^{\rm info}$, and $\hat P_t^{\rm wls}$, respectively. The normalized squared errors corresponding to the three paths are equal to $0.1842, 0.1696$ and $0.1506$, respectively. The much lower estimation error corresponding to $\hat P_t^{\rm wls}$ is because of its capability in tracking rotations in the covariance.
\begin{figure*}[htb]
\centering
\includegraphics[width=.8\textwidth]{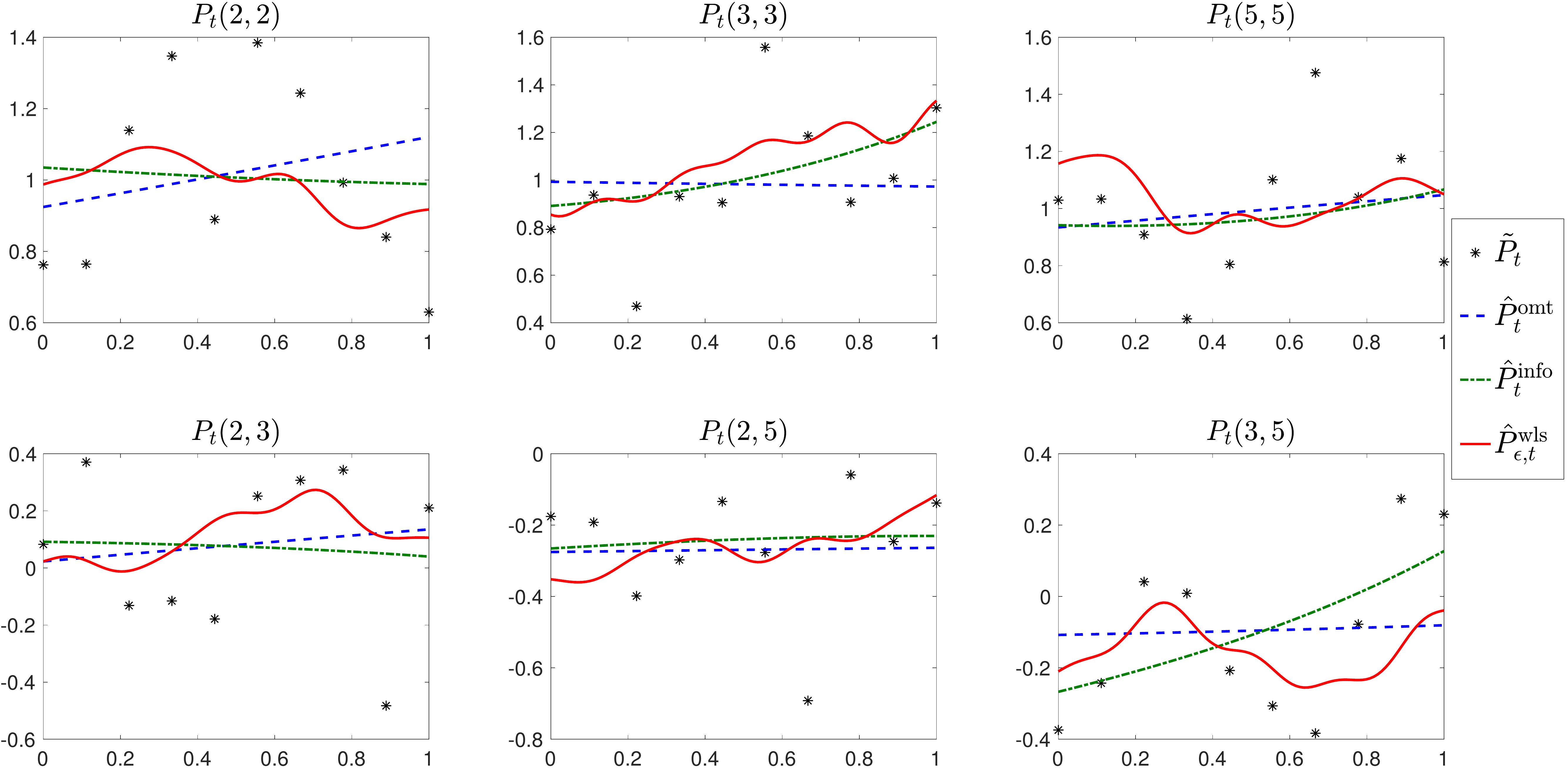}
\caption{\small{These plots illustrate the noisy sample covariance $\tilde P_t$ and the fitting results of three covariance paths based on a rsfMRI dataset.}}\label{fig:Fit}
\end{figure*}

\section{Discussion and conclusion}\label{eq:discussion}
In this paper, we have investigated three types of covariance paths by using different regularizations on linear stochastic systems.
The first covariance path is given by the solution to the Schr\"odinger bridge problem for multivariate Gaussian and optimal mass transport.
The second type of covariance path is based on a generalization of the Fisher-Rao metric with a stochastic input term.
The third type of covariance paths is obtained using a weighted-least-square regularizations on the system matrices.
It is interesting to remark that the Fisher-Rao metric can be viewed as a weight-mass-transport cost.
The corresponding covariance path is a special case of the weighted-least-square based solutions.
The main contributions of this work include the differential-equation formulations for the last two types of covariance paths and the closed-form expressions for the scalar-valued case.

The general theme of this paper is closely related to the work in \cite{JNG2012,Ning2013,NingMatrixOMT,ChenOMT,ChenWas1,Yamamoto2017} which all focused on investigating smooth paths connecting positive definite matrices. The proposed approach is based on different regularizations functions of system matrices which distinguishes this paper from early work. Finally, we remark that a main motivation of this paper is from a neuroimaging application on analyzing functional brain networks using resting-state functional MRI data. The proposed covariance paths with rotating eigenspace may provide a useful tool to understand the oscillations and directed connections among brain networks, which will be further explored in our future work.
\appendix
\balance
{\small
\bibliographystyle{IEEEtran}
\bibliography{IEEEabrv,CovMat}
}
\end{document}